\PassOptionsToPackage{x11names}{xcolor}
\documentclass[12pt,a4paper,twoside]{article}

\usepackage{array}
\usepackage{amsmath}
\usepackage{tabularx}
\usepackage{txfonts}
\usepackage{enumerate}
\usepackage{cite}
\usepackage{float}
\usepackage{xspace}
\usepackage[x11names]{xcolor}

\usepackage{tikz}
\usepgflibrary{shapes}

\usepackage[T1]{fontenc}

\usepackage[margin=2cm,head=15pt]{geometry}

\DeclareMathOperator{\e}{\mathrm{e}}
\DeclareMathOperator{\dd}{\mathrm{d\!}}
\DeclareMathOperator{\Int}{\mathrm{Int}}

\newcommand{\N}{\varmathbb{N}}
\newcommand{\R}{\varmathbb{R}}

\newcommand{\Cb}{\mathcal{C}}

\newcommand*{\kryt}{\textrm{cr}}

\newtheorem{thm}{Theorem}[section]

\newtheorem{prop}[thm]{Proposition}
\newtheorem{cor}[thm]{Corollary}
\newtheorem{rem}[thm]{Remark}
\newtheorem{defs}[thm]{Definition}

\newenvironment{proof}[1][]{\par\noindent\textbf{Proof #1: }}{\hfill\rule{1.3ex}{1.3ex}\par}

\numberwithin{equation}{section}

\usepackage{fancyhdr}
\begin{document}
\begin{center}{\Large\bf 
General model of a~cascade of reactions with time delays: global stability analysis}
\end{center}

\medskip{}
\begin{center}{\large%
Marek Bodnar}

\medskip{}
$^1$ Institute of Applied Mathematics and Mechanics,\\
      University of Warsaw, Banacha 2, 02-097 Warsaw, Poland.\\
		\texttt{mbodnar@mimuw.edu.pl}\\[0.8ex]
\end{center}

\bigskip{}
%
%
%
%
%

%
%

\begin{abstract}
The problem considered in this paper consists of a~cascade of reactions with discrete as well as distributed delays, which arose in the context of Hes1   gene expression. For the abstract general model sufficient conditions for global stability are presented. Then the abstract result is applied to the Hes1 model. 
\end{abstract}


\section{Introduction}

The paper is motivated by the problem of global stability of a~positive steady state
in the system of delay differential equations that describes gene expression of Hes1 protein. 
A~scheme of biochemical reactions connected with this process is shown in the left-hand side 
panel of Fig.~\ref{fig:hes1}. The mathematical model proposed by Monk~\cite{monk03currbiol} 
consists of two ordinary differential equations with time delays that reflect 
protein production time and mRNA~transcription. Local stability of the positive steady state of the model
was extensively studied in literature (eg.~\cite{bernard06philtrans,Zhuang10ijcms,mbab12nonrwa,Erneux09} and references therein). It is known, that  
if degradation rates of Hes1 protein and its mRNA~are sufficiently large, the positive steady state is 
locally asymptotically stable and stability does not depend on time delay (see~\cite{bernard06philtrans,mbab12nonrwa}). 
On the other hand, for other values of degradation rates of Hes1 protein and its mRNA, the stability depends on
the sum of delays in transcription of mRNA~and protein's production and if the delay exceeds some critical 
value, Hopf bifurcation occurs (see~\cite{bernard06philtrans,mbab12nonrwa}). In~\cite{mbab12nonrwa} a~direction 
of the bifurcation was studied, and conditions guaranteeing existence of the supercritical
bifurcation was found. However, we are not aware of any results that address the question of global 
stability of the steady state of this system. In this paper we study 
 global stability of the positive steady state in the Hes1 gene expression model and we generalise the result for similar systems as those shown  in the right-hand side of Fig.~\ref{fig:hes1}. This kind of 
systems can be also considered as a~simple signalling pathway. Gene expression models as well as signalling 
pathways are involved in many more complicated biological phenomena, including carcinogenesis (see 
eg.~\cite{jmjpmbuf11bmb,mbufjp11jmaa,hasty2001nrg,Jong2002jcb,Cerone2012plos} and references therian).
We hope the results as well as the methods presented in the paper could be used to various, even more complicated models, and also to tumour growth models. 

\begin{figure}[bht]
\centerline{
\begin{tabular}{cp{2em}c}
\begin{tikzpicture}[>=stealth,yscale=0.6,xscale=0.55]
	\node (DNA) at (0,0)  {DNA};
	\node (mRNA) at (4,0) {mRNA};
	\node (Hes1) at (4,-6) {Hes1};
   \node (degrna) at (8,0) {$\emptyset$};
   \node (deghes) at (8,-6) {$\emptyset$};
	\begin{scope}[semithick]
		\draw[->,shorten >=2pt, double] (DNA) to node[above] {$\gamma_r$} (mRNA) ;
      \draw[->,shorten >=2pt] (mRNA) to node[above] {$k_r$} (degrna);
		\draw[->,shorten >=2pt, double] (mRNA) to node[right] {$\gamma_p$} (Hes1);
      \draw[->,shorten >=2pt] (Hes1) to node[above] {$k_p$} (deghes);
		\draw[-|,shorten >=2pt,bend left=30] (Hes1) to node[pos=0.7,right] {$f(\textrm{Hes1})$} (DNA);
	\end{scope}
\end{tikzpicture}& &
\begin{tikzpicture}[>=stealth,,yscale=0.6,xscale=0.55]
  	\begin{scope}[semithick, shorten >=2pt]
		\node (X0) at (0,0)  {$A$};
		\node (X1) at (4,0) {$X_1$};
		\node (X2) at (4,-2) {$X_2$};
   	\node (X3) at (4,-4) {$X_3$};
		\node (Xk) at (4,-6) {$X_k$};
  		\node (X1deg) at (8,0) {$\emptyset$};
   	\node (X2deg) at (8,-2) {$\emptyset$};
   	\node (X3deg) at (8,-4) {$\emptyset$};
   	\node (Xkdeg) at (8,-6) {$\emptyset$};
		\draw[->, double] (X0) to node[above] {$\alpha_1$} (X1);
      \draw[->] (X1) to node[above] {$\mu_1$} (X1deg);
		\draw[->, double] (X1) to node[pos=0.4,right] {$\alpha_2$} (X2);
      \draw[->] (X2) to node[above] {$\mu_2$} (X2deg);
		\draw[->, double] (X2) to node[pos=0.4,right] {$\alpha_3$} (X3);
      \draw[->] (X3) to node[above] {$\mu_3$} (X3deg);
		\draw[->, double,dotted] (X3) to node[pos=0.4,right] {$\alpha_k$} (Xk);
      \draw[->] (Xk) to node[above] {$\mu_k$} (Xkdeg);
		\draw[->,shorten >=2pt,bend left=30] (Xk) to node[pos=0.7,right] {$f(X_k)$} (X0);
	\end{scope}
\end{tikzpicture} \\
(a) & & (b)
\end{tabular}
}
\caption{(a) Sketch of negative feedback loop for the Hes1 system.
     (b) Sketch of simple signalling pathway with feedback.\label{fig:hes1}}
\end{figure}
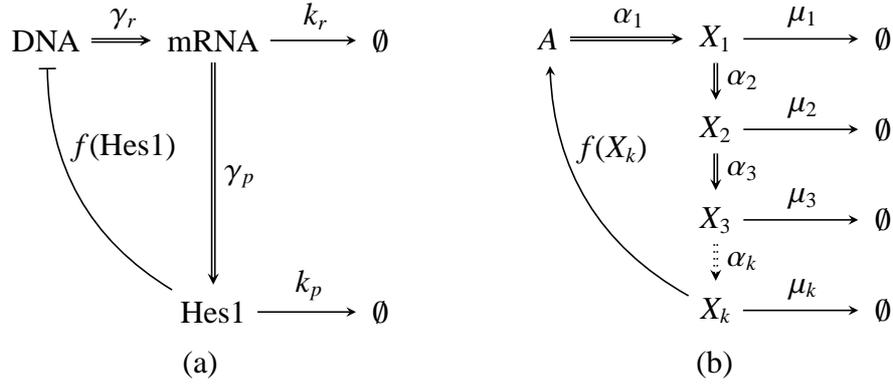

Standard tool used in proving global stability is the method of Liapunov functionals. However, 
although the method is well known, the construction of suitable functional is 
usually a~``bottle neck''. Recently, Liz and Ruiz-Herrera~\cite{Li2010} proposed the method for 
proving global stability of the steady state of delay differential equations by investigating 
the asymptotic behaviour of some corresponding discrete dynamical system. The method was developed 
for Hopfield's model of neural networks. Here, we adapt this method to other type of 
delay differential equations that arises from the model of Hes1 gene expression. The main idea 
of the Liz-Ruiz-Herrera method is to determine global stability of the trivial steady state 
of the equation
\begin{equation}\label{liz}
	\dot x_j(t) = -x_j(t) + F_j\bigl(x_1(t-\tau_{j1}),x_2(t-\tau_{2j}),\dotsc,
			x_k(t-\tau_{jk})\bigr), \quad j=1,2,\dotsc,k,
\end{equation}
where $F\colon \R^k\to\R^k$, $\tau_{j\ell}\ge 0$, 
by assuming suitable asymptotic properties for the discrete system
\begin{equation}\label{liz:dys}
	y_j(n+1) = F_j(y(n)), \quad y=(y_1,\dotsc,y_k), \quad j=1,2,\dotsc,k, \quad n=1,2,3,\dotsc.
\end{equation}
In fact, if we rewrite Eq.~\eqref{liz} as 
\[
	\varepsilon \dot x_j(t) = -x_j(t) + F_j\bigl(x_1(t-\tau_{j1}),x_2(t-\tau_{2j}),\dotsc,
			x_k(t-\tau_{jk})\bigr), \quad j=1,2,\dotsc,k,
\]
and let $\varepsilon\to 0$ we arrive at~\eqref{liz:dys}. Clearly, after time rescaling,
which is equivalent to set new delays $\tilde\tau_{j\ell} = \tau_{j\ell}/\varepsilon\to+\infty$
(see~\cite{MalletParet1986} for extensive study of one singularly perturbed differential
 equation), we may say the Liz-Ruiz-Herrera method investigates the behaviour of~\eqref{liz}
for very large delays. The main weak point of the method is a~strong assumption 
made on the discrete system, which require the steady state to be 
a strong attractor. We precise this notion in the next section in Definition~\ref{def:sa}. 
The second issue is the term ``$-x_j(t)$''. However, it can be easily overcome. 
If the term $-g_j(x_j(t))$ appears instead and $g_j$ is a~homeomorphism, it is enough to consider 
$g_j^{-1}\bigl(F_j(y)\bigr)$ in the right-hand side of the corresponding discrete system (see Remark~2.2 in~\cite{Liz2013jde}). The method works also for particular systems with distributed delays. 

Here we adapt the method proposed in~\cite{Liz2013jde} for the system that arises
in the context of gene expression models. We consider a~cascade of reactions with feedback as shown in the right-hand side panel of Fig.~\ref{fig:hes1}. We formulate 
a general theorem that shows, under suitable assumptions, global stability of the steady 
state of the system. The condition is independent of the magnitude of delay, 
and therefore is limited to the region in the parameter space, where  the steady state
is locally stable independently of the delay. Although the Liapunov functionals method can give 
stronger condition, this method is easier to apply. As an example, application of the 
theoretical result to the Hes1 gene expression model is given. 

The paper is organised as follows. In Section~\ref{sec:main} we formulate an abstract 
framework and prove the general theorem. In Section~\ref{sec:appl} we use the 
results proved in Section~\ref{sec:appl} to the Hes1 gene expression 
model. We conclude the paper with a~short discussion.

\section{Abstract theory}\label{sec:main}

\subsection{Notation and model description}
Let $\tau\ge 0$ be an arbitrary real number. For an arbitrary set $\Omega\subset\R^k$, by 
$\Cb(\Omega)$ we denote the set of continuous 
functions defined on $[-\tau,0]$ with values in $\Omega$ with standard supremum norm.

We consider the following system of equations:
\begin{equation}\label{uklad}
	\begin{split}
		\dot x_j(t) &= \int_{-\tau}^0 \theta_j(s) f_j(t,x_{j-1}(t+s)) \dd{}s - \mu_j x_j(t), \quad j=1,2,\dotsc,k, \;\; 
							j\bmod k,\\
	\end{split}
\end{equation}
where the functions $f_j\colon \R\times\R \to \R$ are continuous fulfilling $f_j(t,0)=0$
for all $j=1,2,\dotsc,k$, $t\in\R$ and $\theta_j$, $j=1,2,\dotsc,k$, are
probabilistic measures, that is  $\theta_j(s)\ge 0$ and  $\int_{-\tau}^0\theta_j(s)ds = 1$. 
To close the system we consider continuous initial condition 
\begin{equation}\label{ic}
	x_j(t_0+s) = \phi_j(s), \quad s\in [-\tau,0], \quad 
	\phi_j\in\Cb(\R), \quad j=1,2,\dotsc,k.
\end{equation}

\begin{thm}\label{thm:ex}
	Assume that the functions $f_j\colon\R\times\R\to\R$ are continuous 
	and fulfil local Lipschitz condition with respect to the second variable. 
	Then for any initial function $\phi=(\phi_1,\dotsc,\phi_k)\in\Cb(\R^k)$
	and arbitrary $t_0\in\R$, 
	there exists a~unique solution to the problem~\eqref{uklad},~\eqref{ic}. 
	Moreover, if for each $j=1,2,\dotsc,k$ the alternative:
	the support of $\theta_j$ is separated from $0$ or $f_j$ is globally Lipschitz 
	continuous holds, then the solution to the problem~\eqref{uklad},~\eqref{ic} 
	is defined for all $t\ge 0$. 
\end{thm}
\begin{proof}
	Local existence follows from the standard existence theorem for delay differential equations 
	(DDEs);~\cite{Hale93}. 

	Define a~set of indexes $I_g\subset \{1,2,\dots,k\}$ such that for any $j\in I_g$
	the function $f_j$ is globally Lipschitz.
	
	If $I_g = \{1,2,\dotsc,k\}$, then global existence follows from the standard theory of DDEs; \cite{Hale93}. 

	Assume now, that $I_s=\{1,2,\dotsc,k\}\setminus I_g\neq\emptyset$. There exists $\tau_{\min{}}>0$ such that 
	supports of $\theta_{j}$ are included in $[-\tau,-\tau_{\min{}}]$ for all $j\in I_s$, that is for 
	all $j\in I_s$, $t\in\R$ and $\varphi\in \Cb(\R)$ we have 
	\[
		\int_{-\tau}^0 \theta(s)f(t,\phi(s))\dd s = \int_{-\tau}^{-\tau_{\min{}}} \theta(s)f(t,\phi(s))\dd s. 
	\]
	Let $t\in [0,\tau_{\min{}}]$.  For any $j\in I_s$ the equation for $x_j$ reads
	\begin{equation}\label{lin:nonhom}
		\dot x_j(t) = \int_{-\tau}^{-\tau_{\min{}}} \theta_j(s) f_j(t,\varphi_{j-1}(t+s)) \dd{}s - \mu_j x_j(t), \;
		j\bmod k.
	\end{equation}
	For each $j\in I_s$, Eq.~\eqref{lin:nonhom} is a~non-homogenous linear equation, so the 
	solution exists on the whole interval $[0,\tau_{\min{}}]$. The right-hand sides of 
	equations for $x_j$, $j\in I_g$, are globally Lipschitz with respect to the second variable, and therefore
	the solution is defined on the whole interval $[0,\tau_{\min{}}]$. 
	The same argument and mathematical induction allows us to prolong the solution  on 
	the interval $[n\tau_{min{}},(n+1)\tau_{\min{}}]$, for any $n\in\N$.
	This completes the proof.
\end{proof}
\begin{rem}
	If the support of $\theta_j$ is separated from zero, then we can relax the assumption that 
	$f_j$ is Lipchitz continuous  and the assertion of Theorem~\ref{thm:ex} remains true 
	for $f_j$ that is only continuous (or even only integrable). 
\end{rem}

In this paper we prove the following general result.
\begin{thm}\label{thm:globalstability}
	Let the assumptions of Theorem~\ref{thm:ex} hold and $f_j(t,0)=0$ for all $j=1,2\dotsc,k$. 
	Assume also that $\mu_j>0$, $j=1,2,\dotsc,k$, and there exist non-negative numbers $\alpha_j$, $j=1,2,\dotsc,k$, 
	such that the functions $f_j$ fulfil
	\begin{equation}\label{warunkig}
			|f_j(t,x)| \le \frac{\alpha_j}{\mu_j}|x|, \quad x\in\R, \;\; j=2,3,\dotsc,k, \qquad 
			|f_1(t,x)| < \frac{\alpha_1}{\mu_1}|x|, \quad x\neq 0, 
	\end{equation}
	for all $t\ge t_0$. 
	If $\alpha_1\alpha_2\dotsm\alpha_k\le\mu_1\mu_2\dotsm\mu_k$,  then the trivial steady state of~\eqref{uklad} 
	is globally asymptotically stable, that is for any initial function
	$\phi=(\phi_1,\dotsc,\phi_k)\in\Cb(\R^k)$ the solution to~\eqref{uklad},~\eqref{ic}  
	converges to $0$ as $t\to+\infty$.
\end{thm}
To prove Theorem~\ref{thm:globalstability} we use Theorem~2.5 from the paper by Liz and Rus-Herrera;~\cite{Liz2013jde}. 

\begin{defs}[Definition 2.1 in \cite{Liz2013jde}]\label{def:sa}
	Let $H\colon D\to D$ be a~continuous map defined on $D = (a_1, b_1)\times (a_2,b_2)\times\dotsm\times(a_k,b_k)$.
	An equilibrium $y^*\in D$ of the system
	\[
		y(n+1) = H\bigl(y(n)\bigr), \quad  n=1,2,3,\dotsc, 
	\]
	is a~strong attractor in $D$ if for every compact set $K\subset D$ there exists a~family of sets $\{I_m\}$, $m\in\N$,
	where $I_m$ is the product of $k$ nonempty compact intervals, satisfying
	\begin{enumerate}[(B1)]
		\item  $K \subset \Int(I_1)\subset  D$,
		\item  $H(I_m) \subset I_{m+1} \subset \Int(I_m)$ for all $m\in\N$,
		\item  $y^*\in \Int(I_m)$ for all $m\in\N$, and $\bigcap_{m=1}^\infty I_m = \{ y^*\}$.
	\end{enumerate}
\end{defs}

First, we prove the following.
\begin{thm}\label{thm:golbal}
	Let $h_j:\R\to\R$, $j=1,2,\dotsc,k$, be arbitrary continuous functions, such that 
	$|h_1(x)|<\beta_1|x|$ for $x\neq 0$, and $|h_j(x)| \le \beta_j |x|$, $j=2,3,\dotsc,k$, 
	for all $x\in\R$. 
	If $\beta_1\beta_2\dotsm\beta_n\le 1$, then the point $0\in\R^k$ is a~strong attractor 
	in $\R^k$ of the discrete dynamical system 
	\[
		y_{j}(n+1) = h_j(y_{j-1}(n)), \quad j=1,2,3,\dotsc,k, \; j \bmod k, 
				\qquad n=1,2,\dotsc
	\]
\end{thm}
\begin{proof}
	Define an operator $H\colon \R^k\to \R^k$,
	\[
		H(y_1,y_2\dotsc,y_k) = \Bigl( h_1(y_k), h_2(y_1), \dotsc, h_k(y_{k-1})\Bigl)\,.
	\]
	Let $K\subset \R^k$ be an arbitrary compact set. 
	First, we prove the assertion of Theorem~\ref{thm:globalstability} under the assumption $\beta_1\beta_2\dotsm\beta_k<1$. 
	Next, we explain how to adapt the arguments to the case $\beta_1\beta_2\dotsm\beta_k=1$. Although arguments 
	are very similar in both cases, in the case $\beta_1\beta_2\dotsm\beta_k<1$ we can omit some technicalities 
	and we belive it is easier to follow the main idea of the proof. 

	Assume $\beta_1\beta_2\dotsm\beta_k<1$. 
	We chose a~positive numbers $q_j$, $j=1,2,\dotsc,k-1$, in a~recursive manner. Let $q_k=1$ and let 
	$q_j$, $j=1,2,\dotsc,k-1$ be any numbers satisfying 
	\begin{equation}\label{defq}
			\beta_1\beta_2\dotsm\beta_j <q_j <\frac{q_{j+1}}{\beta_{j+1}}, \quad j=1,2,\dotsc,k-1.
	\end{equation} 
	Indeed, such numbers exists. We prove it by induction. As $\beta_1\beta_2\dotsm\beta_k<1$, we have 
	$q_k = \beta_1\beta_2\dotsm\beta_k+\varepsilon_k$, with $\varepsilon_k>0$. Now, assume 
	that $q_{j+1} = \beta_1\beta_2\dotsm\beta_j\beta_{j+1}+\varepsilon_{j+1}$ and choose 
	$\varepsilon_j$ as an arbitrary number such that 
	\[
		0<\varepsilon_j<\frac{\varepsilon_{j+1}}{\beta_{j+1}}. 
	\]
	It is easy to check, that~\eqref{defq} holds for $q_j=\beta_1\beta_2\dotsm\beta_j+\varepsilon_j$. 
	Hence, by mathematical induction there exists $q_j$, $j=1,2,\dotsc,k-1$, such that 
	 inequalities~\eqref{defq} hold. 

	Set $a$ so large that 
	\[
		K \subset  [-q_1 a,q_1a]\times [-q_2 a,q_2a] \times \dotsm 
								[-q_{k-1} a,q_{k-1} a]\times [-a,a]. 
	\]
	Define
	\begin{equation}\label{aj1}
		a_j(1) = q_j a, \;\; j=1,2,\dotsc, k-1, 
	\end{equation}
	and
	\begin{equation}\label{ajm}
		a_j(m+1) = \beta_j a_{j-1}(m), \quad j=1,2,\dotsc,k,\;j\bmod k, \qquad  m=1,2,3,\dotsc,
	\end{equation}
	and
	\[
		I_m = [-a_1(m),a_1(m)]\times[-a_2(m),a_2(m)]\times\dotsm\times[-a_k(m),a_k(m)], \quad m=1,2,3,\dotsc.
	\]
	Due to~\eqref{defq}, \eqref{aj1} and \eqref{ajm} we have 
	\[
		a_j(2) = \beta_j a_{j-1}(1) = \beta_j q_{j-1} a~< q_j a~= a_j(1), \;\;j=1,2,\dotsc,k,\;\; j\bmod k,
	\]
	and therefore $H(I_1)\subset I_2\subset \Int(I_1)$. 
	For any $m=2,3,\dotsc$, and $j=1,2,\dotsc,k$, $j\bmod k$, the following implication is true 
	\begin{equation}\label{implikacja}
		a_{j-1}(m-1)>a_{j-1}(m) \; \Longrightarrow \;
		a_j(m+1) = \beta_j a_{j-1}(m) < \beta_j a_{j-1}(m-1) = a_j(m).
	\end{equation}
	Thus, if $I_m\subset I_{m-1}$, then $I_{m+1}\subset I_m$ and the claim 
	that $I_{m+1} \subset \Int(I_m)$ for all $m\in\N$ is proved by mathematical induction.
	Moreover, the assumption $|h_j(x)|\le\beta_j|x|$ implies $H(I_m)\subset I_{m+1}$ for all $m\in\N$,
	and hence the family of sets $\{I_m\}$ fulfil conditions (B1) and (B2) of Definition~\ref{def:sa}. 

	It remains to prove that condition (B3)  of Definition~\ref{def:sa} holds. 
	It is clear that $0\in I_m$ for all $m\in\N$. 
	Note, because $\beta_1\beta_2\dotsm\beta_k<1$,
	\[
		a_1(k\ell+j) = \Bigl(\beta_1\beta_2\dotsm\beta_k\Bigr)^{\ell}a_1(j) \to 0, \;\;\text{ for any } \;\;
			j=1,2,\dotsc,k, \;\; \ell\in \N,
	\]
	as $\ell\to+\infty$. This implies $a_1(m)\to 0$ as $m\to+\infty$. On the other hand, 
	$a_j(m) = \beta_2\beta_3\dotsm\beta_j a_1(m)$, $j=2,3,\dotsc,k$. Thus, for any
	$j=1,2,\dotsc,k$ we have $a_j(m)\to 0$ as $m\to+\infty$, so 
	\[
		\bigcap_{m=1}^{\infty} I_m = \{ 0 \}, 
	\]
	condition (B3)  of Definition~\ref{def:sa} holds, and the 
	assertion of Theorem~\ref{thm:golbal} is proved for $\beta_1\beta_2\dotsm\beta_k<1$.

	\smallskip
	It remains to prove the assertion if $\beta_1\beta_2\dotsm\beta_k=1$. Let $\tilde h$ be 
	an arbitrary continuous increasing function such that for all $x> 0$ inequalities 
	\begin{equation}\label{defhtilde}
		|h_1(x)| < \tilde h_1(x) <\beta_1 |x|
	\end{equation}
	hold. Such function exist due to the assumption $|h_1(x)|<\beta_1|x|$ and that 
	the function $x\mapsto\beta_1 |x|$ is increasing. 
	Set an arbitrary $\tilde a>0$
	such that $K\subset [-\tilde a,\tilde a]^k$ and 
	\[
		r = \min_{j=1,2,\dotsc,k} \{ \beta_1\beta_2\dotsm\beta_j\}\le 1. 
	\]
	Take $a = \tilde a/r$. Let us choose numbers $q_j$ such that inequalities~\eqref{defq}
	hold but with $\beta_1$ replaced by $\tilde h_1(a)/a$, and define $a_1(j)$,
	$j=1,2,\dots,k$ as in~\eqref{aj1}. Now, we define $a_j(m)$ for $j=2,3,\dotsc,k$, and $m=2,3,4,\dotsc$ as 
	in~\eqref{ajm}, and take $a_1(m) = \tilde h_1(a_k(m-1))$.
	Note, that under such choice of $a_j(m)$, $j=1,2,\dotsc,k$ and $m=2,3,4,\dotsc$, the
	implication~\eqref{implikacja} also holds. Indeed, the argument for $j=2,3,\dotsc,k$
	is the same as above, and for $j=1$ we have 
	\[
		a_1(m) = \tilde h_1(a_k(m-1)) > \tilde h_1(a_k(m)) = a_1(m+1),
	\]
	due to~\eqref{defhtilde} and the fact that  $\tilde h$ is increasing.
	Thus, the arguments as in the case  $\beta_1\beta_2\dotsm\beta_k<1$  and 
	the first inequality of~\eqref{defhtilde} yield $H(I_m)\subset I_{m+1} \subset \Int(I_m)$. 

	To show the intersection of the family $\{I_m\}$ is the point $\{0\}$ we need a~little finner argument than above, 
	but the idea remains similar. Indeed, for any $j=1,2,\dotsc,k$ the sequence 
	$\{a_1(k \ell+j)\}_{l=1}^\infty$ is given by the following equality
	\[
		a_1(k(\ell+1)+j) = \tilde h_1\Bigl(\beta_2\beta_3\dotsm\beta_k a_1(k \ell+j)\Bigr).
	\]
	Because, due to~\eqref{defhtilde}, for any $x>0$
	\[
		0<\tilde h_1\Bigl(\beta_2\beta_3\dotsm\beta_k x\Bigr) < \beta_1 \beta_2\dotsm\beta_k x \leq x ,
	\]
	for any fixed $j=1,2,\dotsc,k$ the sequence $\{a_1(k \ell+j)\}_{l=1}^\infty$ is monotone and bounded, yielding the existence of limit. It is easy to see that this limit is equal to $0$.  The argument as in the case 
	$\beta_1\beta_2\dotsm\beta_k<1$  completes the proof.
\end{proof}

\smallskip{}
\begin{proof}[of Theorem~\ref{thm:globalstability}]
	It is easy to see that the functions $h_j(x) = \frac{1}{\mu_j}f_j(x)$ fulfil 
	assumptions of Theorem~\ref{thm:golbal}. Thus, 
	the assertion of Theorem~\ref{thm:globalstability} follows from
	Remark~2.2 and Theorems~2.5 and 2.6 from~\cite{Liz2013jde}.
\end{proof}

In Theorem~\ref{thm:globalstability} we assumed that the functions $f_j$, $j=1,2,\dotsc,k$, are 
defined on the whole real line and the point $0\in\R^k$ is a~steady state. Of course, 
with a~simple change of variables we can always move a~steady state to $0\in\R^k$. Now, we 
present an obvious corollary that extends our result for the system where the functions
$f_j$ are defined only on some interval. 
\begin{cor}\label{wn:mniejszy}
	The assertion of  Theorem~\ref{thm:globalstability} remains true with 
	the functions $f_j$ defined on the closed intervals $J_j$ if the set 
	$J_1\times J_2\times\dotsm\times J_k$ is invariant under the evolution 
	of the system~\eqref{uklad}.
\end{cor}
\begin{proof}
	If the function $f_j$ is defined on the interval $[a_j,b_j]$,  we 
	extend it on the whole line setting $f_j(x)=f_j(a_j)$ for $x<a_j$
	and $f_j(x)=f_j(b_j)$ for $x>b_j$.  If the function $f_j$ is defined on
	$(-\infty,b_j]$ or on $[a_j,+\infty)$ we proceed analogously. 
	After extending all $f_j$, $j=1,2,\dotsc,k$ on the whole line we can apply
	Theorem~\ref{thm:globalstability}. 
\end{proof}

The condition in~Theorem~\ref{thm:globalstability} is sharp in the following sense. If 
$f_j$ are differentiable at $0$, and assumptions of Theorem~\ref{thm:globalstability}
hold, then $|f_j'(0)|\le\alpha_j$. Hence, we can formulate the following result
\begin{prop}\label{prop:sharp}
	Let $f_j$, $j=1,2,\dotsc,k$, be differentiable at $0$ and
	assume that the hypothesis of Theorem~\ref{thm:globalstability} holds
	with   $\alpha_j=|\gamma_j|$, where $\gamma_j = f_j'(0)$, $j=1,2,\dotsc,k$. Then the
	following statements are true 
	\begin{enumerate}[(i)]
		\item If $\gamma_1\gamma_2\dotsm\gamma_k>0$, then the trivial steady state is 
			  globally asymptotically stable for $\gamma_1\gamma_2\dotsm\gamma_k\le\mu_1\mu_2\dotsm\mu_k$ 
			  and unstable otherwise.
		\item If $\gamma_1\gamma_2\dotsm\gamma_k<0$ 
			  and $|\gamma_1\gamma_2\dotsm\gamma_k|\le\mu_1\mu_2\dotsm\mu_k$, 
			  then the trivial steady state is globally asymptotically stable.
		\item If $\gamma_1\gamma_2\dotsm\gamma_k<0$ 
			  and $|\gamma_1\gamma_2\dotsm\gamma_k|>\mu_1\mu_2\dotsm\mu_k$, 
			  each of the measures $\theta_j$ is concentrated in a~single point $\tau_j$,
			  then one of two possibilities can occur.
			  If the zero steady state is locally asymptotically stable 
			  for $\tau=\tau_1+\tau_2+\dotsb+\tau_k=0$, then there exists a~critical 
			  value $\tau_{\kryt}>0$, such that the trivial steady state is stable 
			  for $\tau<\tau_{\kryt}$, unstable for $\tau>\tau_{\kryt}$, and at 
			  the point $\tau=\tau_{\kryt}$ Hopf bifurcation occurs. 
			  On the other hand, if the zero steady state is unstable for 
			  $\tau=\tau_1+\tau_2+\dotsb+\tau_k=0$, then it is unstable for all
			  $\tau>0$. 
	\end{enumerate}
\end{prop}
\begin{proof}
	Stability assertion from points (i) and (ii) follows from Theorem~\ref{thm:globalstability} directly. 

	The characteristic matrix for the trivial
	steady state of~\eqref{uklad} reads
	\[
		\begin{bmatrix}
			\lambda+\mu_1 & 0 & 0 & \ldots & 0 & -\gamma_1  \eta_1(\lambda) \\
			-\gamma_2 \eta_2(\lambda) & \lambda+\mu_2 & 0 &\ldots & 0 & 0 \\
			0 & -\gamma_3\eta_3(\lambda) & \lambda+\mu_3 & \ldots & 0 & 0 \\
			\vdots & \vdots & \vdots & \ddots & \vdots  &\vdots \\
			0 & 0& 0& \ldots & -\gamma_k\eta_k(\lambda) & \lambda+\mu_k
		\end{bmatrix},
	\]
	where
	\[
		\eta_j(\lambda) = \int_{-\tau}^0 \theta(s)\e^{-s\lambda}\dd s.
	\]
	The characteristic function reads
	\[
		W(\lambda) = \prod_{j=1}^k (\lambda+\mu_j) - \prod_{j=1}^k \gamma_j\eta_j(\lambda).
	\]
	Note, if $\gamma_1\gamma_2\dotsm\gamma_k>0$ and $\gamma_1\gamma_2\dotsm\gamma_k\le\mu_1\mu_2\dotsm\mu_k$ 
	we have 
	$W(0) <0$ because $\eta_j(0)=1$, $j=1,2,\dotsc,k$. In the case without delay (that is if 
	all $\theta_j$ are concentrated at $0$), 
	we immediately deduce that $W$ has a~positive real root and the trivial steady state is unstable. 
	For the case with delay we use the Mikhailov Criterion (see~\cite{uf04jbs}). To get stability we require
	that the change of argument of the vector $W(i\omega)$ while $\omega$ changes from $0$ to $+\infty$ is
	equal to $k\pi/2$, which is impossible because $W(0)<0$ and argument of $W(i\omega)$ tends to $k\pi/2$ 
	as $\omega\to+\infty$.
	
	Now, we consider the case $\gamma_1\gamma_2\dotsm\gamma_k<0$ and 
	$|\gamma_1\gamma_2\dotsm\gamma_k|>\mu_1\mu_2\dotsm\mu_k$.  
	If each of the measures $\theta_j$, $j=1,2,\dotsc,k$, is concentrated at the single point $\tau_j$, then 
	$\eta_j(\lambda) = \e^{-\lambda\tau_i}$ and 
	\begin{equation}\label{char:fun}
		W(\lambda) =(\lambda+\mu_1)(\lambda+\mu_2)\dotsm(\lambda+\mu_k) - 
					\gamma_1\gamma_2\dotsm\gamma_k\e^{-\lambda\tau}.
	\end{equation}
	Note, if stability change occurs, there exists a~purely imaginary root $i\omega_0$ of $W(\lambda)$,
	with $\omega_0\ge 0$. This is equivalent to existence of a~
	positive root of the function 
	\begin{equation}\label{fun:pom:F}
  	  \begin{split}
		F(\omega) &= \bigl|(i\omega+\mu_1)(i\omega+\mu_2)\dotsm(i\omega+\mu_k)\bigr|^2 - 
			\bigl(\gamma_1\gamma_2\dotsm\gamma_k\bigr)^2\\
			&=\bigl(\omega^2+\mu_1^2\bigr)\bigl(\omega^2+\mu_2^2\bigr)\dotsm
			  \bigl(\omega^2+\mu_k^2\bigr) -\bigl(\gamma_1\gamma_2\dotsm\gamma_k\bigr)^2
	  \end{split}
	\end{equation}
	The function $F$ is a~polynomial of the degree $2k$. It is easy to see that it 
	has exactly one positive root $\omega_0>0$ if $|\gamma_1\gamma_2\dotsm\gamma_k|>\mu_1\mu_2\dotsm\mu_k$.
	Indeed, $F(0)<0$ and $F'(\omega)>0$ for all $\omega\ge0$.
	Proposition~1 from~\cite{cook86ekvacioj} ensures that if $F'(\omega_0)>0$, eigenvalues
	cross imaginary axis from left to right. Thus, if the steady state is unstable
	for $\tau=0$, it remains unstable for all $\tau>0$ and if it is stable for $\tau=0$, it 
	loses its stability at some point $\tau_{\kryt}>0$ and remains unstable. 
	Proposition~1 from~\cite{cook86ekvacioj} implies also that in this case the Hopf bifurcation
	occurs. 
\end{proof}
%

\section{Applications} \label{sec:appl}

In this section we consider two particular examples of models to which the theory developed 
in the previous section can be applied. First, we consider a~signalling pathway with feedback. 
We assume that a~sequence of reactions is such that each one triggers the next one and the last 
one suppresses (or inhibits) an external production of the first chemical 
(see left-hand side panel of Fig.~\ref{fig:hes1} for the scheme of reactions). The second 
example would be the model of Hes1 gene expression proposed by Monk~\cite{monk03currbiol}. 

\subsection{Signalling pathway model} 

Let $\alpha_j$, $j=2,3\dotsc,k$, be positive constants and $f: [0,+\infty)\to\R^+$
be Lipschitz continuous and bounded. Consider the model of $k$ reactions such that 
the chemical $X_j$ induces production of the chemical $X_{j+1}$, $j=1,2,\dotsc,k-1$, while 
the chemical $X_k$ affects  production of the chemical $X_1$ according to the function $f$ (increasing 
function $f$ models activation, while decreasing one models inhibition). We assume each 
reaction is affected by distributed time delay described by the distribution $\theta_j$,
$j=1,2,\dotsc,k$. Denoting by $x_j$, $j=1,2,\dotsc,k$, concentrations of 
the chemicals $X_j$, we arrive at the following model
\begin{equation}\label{model:pathway}
	\begin{split} 
		\dot x_1(t) &= \int_{-\tau}^0 \theta_0(s)f(x_k(t+s))\dd s - x_1(t) , \\
		\dot x_j(t) &= \alpha_j\int_{-\tau}^0 \theta_j(s) x_{j-1}(t+s)\dd s - \mu_j x_j(t), \;\; j=2,3,\dotsc,k,
	\end{split}
\end{equation}
where $\mu_j$ are degradation rates for the chemicals $X_j$, and $\alpha_j>0$ 
denote production rates. We do not lose generality by assuming 
$\mu_1=1$, since we can always rescale time in an appropriate manner. We close the system 
by imposing initial condition 
\begin{equation}\label{ic:pathway}
	x_j(s) = \varphi_j(s)\ge 0, \;\;\;\; s\in [-\tau,0],\;\;\;\;  \varphi\in\Cb(\R), \;\;\;\; j=1,2,\dotsc,k. 
\end{equation}

\begin{prop}\label{prop:positiv}
	For any positive bounded and globally Lipschitz continuous function $f$ there exists 
	a unique solution to~\eqref{model:pathway} with initial data~\eqref{ic:pathway}, which is non-negative 
	and defined for all  $t\geq 0$. 
\end{prop}
\begin{proof}
	Theorem~\ref{thm:ex} implies existence and uniqueness. It remains to prove non-negativity. However, 
	due to positivity of $f$ and $\theta_j$ we can estimate $\dot x_j \ge -\mu_j x_j$ which
	yields desired assertion. 
\end{proof}

An easy observation is the following. 
\begin{prop}\label{prop:ss}
	Let 
	\begin{equation}\label{defdelta}
		\delta_k = 1,\quad \delta_j = \delta_{j+1}\frac{\mu_{j+1}}{\alpha_{j+1}}, 
		\;\; j=1,2,\dotsc,k-1.
	\end{equation}
	If $X = (\bar x_1, \bar x_2, \dotsc, \bar x_k)$ is a~positive steady state of the 
	system~\eqref{model:pathway}, then $\bar x_j$, $j=1,2,\dotsc,k$, fulfil
	\begin{equation}\label{rown:ss}
		f(\bar x_k ) = \delta_1 \bar x_k,\quad 
		\bar x_j = \delta_{j}\bar x_k, \;\; j=1,2,\dotsc,k-1,
	\end{equation}	
\end{prop}
As we can always rescale the $k$-th variable by $\bar x_k$, 
without loss of generality we may assume that $\bar x_k=1$. 
Then the steady state is of the form
\begin{equation}\label{ss}
	X = \bigl( \delta_1,\; \delta_2, \dotsc, \delta_{k-1}, 1\bigr),
\end{equation}
where $\delta_j$, $j=1,2,\dotsc,k-1$, are given by~\eqref{defdelta}. 
\begin{prop}\label{prop:model:pathway}
	Let us assume the hypothesis of Proposition~\ref{prop:positiv} hold. Moreover, 
	assume that $f(1) = \delta_1$. If the point $X\in\R^{k}$ given by~\eqref{ss} is 
	the unique steady state of~\eqref{model:pathway}, 
	\[
		\bigl|f(x)-f(1)\bigr| < \alpha_1 |x - 1| , \quad \text{ for } \quad x>0, \;\; x\neq 1,
	\]
	and $\alpha_1\le \delta_1$, 
	then the steady state $X$ is globally asymptotically stable in 
	$\Cb\bigl([0,+\infty)^{k}\bigr)$. 
\end{prop}
\begin{proof}
	Note, that the set $[0,+\infty)^{k}$ is invariant with respect to the evolution of
	system~\eqref{model:pathway}. As the inequality $\alpha_1\le \delta_1$ is equivalent to 
	$\alpha_1\alpha_2\dotsm\alpha_k\le \mu_2\mu_3\dotsm\mu_k$, the assertion of Proposition~\ref{prop:model:pathway} is a~direct 
	consequence of Corollary~\ref{wn:mniejszy}. 
\end{proof}

\begin{prop}\label{prop:mal}
	Assume that the function $f$ fulfils the following conditions
	\begin{enumerate}[({A}1)]
		\item $f$ is a~$C^2$-class decreasing function, $f(x)>0$ for all $x\in[0,+\infty)$;
		\item there exists $x_c\ge 0$ such that $f$ is convex for $x>x_c$ and 
			it is concave for $0<x<x_c$;
		\item $f(1) = \delta_1$.
	\end{enumerate}
	Define the function 
	\[
		g(x) = (x-1)f'(x) - \bigl(f(x)-f(1)\bigr).
	\]
	Let $x_0$ be defined in the following way. If $g(0)>0$, then $x_0=0$.
	If $g(0)\le 0$ and $x_c=1$ then $x_0=1$, while if $x_c\neq 1$, then 
	$x_0$ is a~positive root of $g(x)$ different from $1$. Then the following 
	statements are true
	\begin{enumerate}[(i)]
		\item There exists exactly one positive steady state of the system~\eqref{model:pathway} given by~\eqref{ss}.
		\item The value of $x_0$ is uniquely defined.
		\item If one of the following statements 
			holds
		\begin{enumerate}[(a)]
			\item $x_0>0$ and $|f'(x_0)|<\delta_1$; 
			\item $x_0 = x_c$ and $|f'(x_0)|\le\delta_1$; 
			\item $x_0=0$ and $f(0)-f(1)<\delta_1$, 
		\end{enumerate}
 then the steady state $X$ is globally asymptotically stable in 
			$\Cb\bigl([0,+\infty)^k\bigr)$
	\end{enumerate}
\end{prop}
\begin{proof}
	Note, the right hand-side of Eq.~\eqref{rown:ss} is an increasing function of $x$, which
	is $0$ at $x=0$, and tends to $+\infty$ as $x\to+\infty$, while 
	$f$ is a~decreasing function of $x$. Obviously $f(0)>0$, which implies assertion 
	of the point (i). 
	
	The first derivative of $g$ reads
	\[
		g'(x) = (x-1)f''(x).
	\]
	If $x_c=0$, then $f$ is convex for all $x>0$ and therefore $g$ is decreasing on 
	$(0,1)$ and increasing on $(1,+\infty)$, so $g(0)>0$ and $g$ has no positive zeros different from $x=1$. 
	Similarly, if $x_c=1$ $g$ also has no positive zeros different from $x=1$.
	Assume now, $x_c>0$ and $x_c\neq 0$. 
	Due to the assumptions of the proposition, $g'(x)$ has exactly two positive 
	roots: $x=1$, and $x=x_c$, and $g'(0)>0$ because $f$ is concave at $x=0$. 
	As $x_c>0$ is an inflection point of $f$,  $g'(x)>0$ for 
	$x\in\Bigl[0,\min\{x_c,1\}\Bigr]\cup\Bigl[\max\{x_c,1\},+\infty\Bigr)$. Assume $0<x_c<1$.
	The function $g$ is decreasing on $(x_c,1)$, increasing on $(0,x_c)\cup(1,+\infty)$, 
	and $g(1)=0$. Therefore, there is no root of $g$ on $(x_c,1)\cup(1,+\infty)$. Thus, if 
	$g(0)\le0$, then there exists exactly one root of $g$ different from $1$, $x_0\in[0,x_c)$. For
	$x_c>1$, similar argument yields that there exists exactly one root of $g$, $x_0>x_c$. 
	This completes the proof of part (ii). 
	
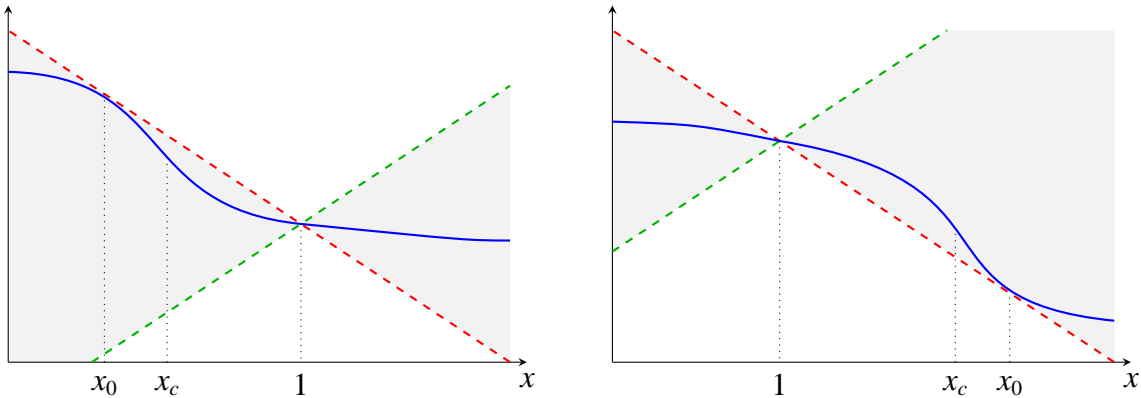
\begin{figure}[!hbt]
\centerline{
	\begin{tikzpicture}[>=stealth,scale=1.1]
		\fill[gray!10!white] (0,0) -- (0,4) -- (6,0) -- (6,3.333) -- (1,0) --cycle;
		\draw[->] (0,0) -- (6.2,0) node[below] {$x$};
		\draw[->] (0,0) -- (0,4.3);
		\draw[dashed,red,thick] (6,0) -- (0,4);
		\draw[dashed,thick,green!70!black] (1,0) -- (6,3.3333);
		\draw[blue,thick,smooth] (0,3.5) .. controls (2,3.45) and (1.5,1.8667) .. (3.5,1.6667) ..
				controls (5.5,1.4667) .. (6,1.4667);
		\draw[dotted] (3.5,1.6667) -- +(0,-1.6667) node[below] {$1$};
		\draw[dotted] (1.15,0) node[below=1pt] {$x_0$} -- +(0,3.233);
		\draw[dotted] (1.9,0) node[below=1pt] {$x_c$} -- +(0,2.5);
	\end{tikzpicture}
	\;\quad \;
	\begin{tikzpicture}[>=stealth,scale=1.1]
		\fill[gray!10!white] (0,1.3333) -- (0,4) -- (6,0) -- (6,4) -- (4,4)  --cycle;	
		\draw[->] (0,0) -- (6.2,0) node[below] {$x$};
		\draw[->] (0,0) -- (0,4.3);
		\draw[dashed,red,thick] (6,0) -- (0,4);
		\draw[dashed,thick,green!70!black] (0,1.3333) -- (4,4);
		\draw[blue,thick,smooth] (0,2.9) .. controls (1,2.8667) .. (2,2.6667) ..
		controls (5,2.1667) and (3.5,0.75) .. (6,0.5);
		\draw[dotted] (2,2.6667) -- +(0,-2.6667) node[below] {$1$};
		\draw[dotted] (4.75,0) node[below=1pt] {$x_0$} -- +(0,0.8333);
		\draw[dotted] (4.1,0) node[below=1pt] {$x_c$} -- +(0,1.6);
	\end{tikzpicture}
}
\caption{Sketch of the graph of the function $f$ (solid line) 
	and a~cone $y<|f'(x_0)||x-1|$ (grey area).
	Cases: $0<x_c<1$ (left) and $1<x_c$ (right).\label{fig.ogr}}
\end{figure}		
	
	Now, we prove part (iii). We find the slope $\alpha_1$ of the cone consisting 
	the graph of the function $f$ and we apply Proposition~\ref{prop:model:pathway}.
	
	If $x_0>0$ (point (iii.a)), as $g(x_0)=0$, the straight line passing through $(1,f(1))$ and with
	the slope 
	$f'(x_0)$ is tangent to the graph of the function $f$ at the point $(x_0,f(x_0))$. 
	Due to convexity assumptions on $f$, this line is above the graph of $f$ for $x<1$
	and below it for $x>1$ (see Fig.~\ref{fig.ogr}). Proposition~\ref{prop:model:pathway}
	with any $\alpha_1>|f'(x_0)|$,  yields 	global stability of the steady state.  

	If $x_c=1$ (point (iii.b)), then $f$ is concave for $x<1$ and convex for $x>1$. This, 
	together with the fact that $f$ is decreasing, implies 
	$|f(x)-f(1)|<|f'(1)||x-1|$ and Proposition~\ref{prop:model:pathway}
	with $\alpha_1=|f'(1)|$ yields global stability of the steady state. 
	
	If $x_0=0$ (point (iii.c)), then the function $f$ is convex for all $x>0$. The line passing 
	through the points $(0,f(0))$ and $(1,f(1))$ has the slope $\alpha_1=f(0)-f(1)$. Moreover, 
	because of convexity of $f$, the graph of $f$ is below it for $x\in(0,1)$ and 
	above for $x>1$. Hence, Proposition~\ref{prop:model:pathway} yields
	global stability of the steady state and this completes the proof.
\end{proof}

\subsection{Hes1 gene expression model}
The model of Hes1 gene expression proposed in 2003 by Monk~\cite{monk03currbiol} 
reads
\begin{equation}\label{hes_sys_gen}
	\begin{split}
		\dot r(\tilde t) &= \tilde f(p(\tilde t-\tau_r)) - k_r r(\tilde t)\,,\\
		\dot p(\tilde t) &= \beta r(\tilde t) - k_p p(\tilde t)\,,
	\end{split}
\end{equation}
where $p$ and $r$ are concentrations of Hes1 and its mRNA, respectively, and 
$\tilde f$ is a~non-increasing, non-negative $C^1([0,+\infty),\R)$ class function, that 
describes negative feedback loop.
Parameters $1/k_r$ and $1/k_p$ are characteristic times for degradation
of mRNA~and Hes1 protein, respectively --- they can be also considered as mean life times of these 
molecules. Parameter $\beta$ is the protein production rate. 

For an arbitrary function $\tilde f$, after a~proper rescaling, the model~\eqref{hes_sys_gen}
is a~particular version of~\eqref{model:pathway}, and as $f$ is non-increasing, 
Proposition~\ref{prop:mal} can be used with $k=2$. Here, using this proposition, 
we derive global stability conditions for the particular type of function used in the 
literature in this context (see~ \cite{jensen03febslett}), that is Hill function.

Note, that $\tilde f$ is non-increasing and this implies that the equation
\begin{equation}\label{nastac}
\tilde f(\xi) - \frac{k_p\: k_r}{\beta} \xi = 0
\end{equation}
has a~unique positive solution. Let denote it by $\bar p$. Now, we introduce the following change of variables
\begin{equation}\label{zamiana}
	\begin{split}
		f(\xi) &= \frac{\beta}{\bar p k_r^2}\, \tilde f(\bar p \xi) \;\textrm{ for all }\; \xi\in[0,+\infty)\,,\qquad
		x(t) = \frac{\beta}{\bar p k_r} r(t), \\
		y(t) &= \frac{1}{\bar p} p(t), \qquad 
		\mu = \frac{k_p}{k_r},\qquad 
		t = k_r \tilde t, \qquad 
		\tau = k_r \tau_r,
	\end{split}
\end{equation}
where $\bar p$ is a~unique positive solution to~\eqref{nastac}. With the change of variables~\eqref{zamiana},
and allowing time delay to be distributed one, the system~\eqref{hes_sys_gen} reads
\begin{equation}\label{hes1dimles}
	\begin{split}
		\dot x(t) &= \int_{-\tau}^0\theta(s)f(y(t+s))\dd s -  x(t)\,,\\
		\dot y(t) &= x(t) - \mu y(t)\,,
	\end{split}
\end{equation}
where $\theta$ is a~probabilistic measure.  We can directly apply
Proposition~\ref{prop:model:pathway}  to the system~\eqref{hes1dimles}. 

Note, that introducing distributed delay, due to the assumption that $\int_{-\tau}^0\theta(s)\dd s =1$, we do not influence existence and the value of the steady state. 

Now, we study the particular case of $\tilde f$ considered in~\cite{jensen03febslett}, namely, 
\begin{equation}\label{fun_jen}
\tilde f(\xi) = \frac{\alpha k^h}{k^h + {\xi}^h},\quad \alpha, k>0, \quad h>1.
\end{equation}
Using definitions~\eqref{fun_jen} and \eqref{zamiana}
and the identity $\bar p = \tfrac{\beta}{k_pk_r} \tfrac{\alpha \bar p^h}{k^h+\bar p^h}$,  simple algebraic calculations lead to
\[
f(\xi) = \frac{k_p}{k_r}\cdot\frac{\left(\frac{k}{\bar p}\right)^h+1}{\left(\frac{k}{\bar p}\right)^h+{\xi}^h} =
     \mu\cdot\frac{b^h+1}{b^h+{\xi}^h}, \quad
   b=\frac{k}{\bar p}.
\]
Calculating the first and the second derivative we obtain
\[
	\begin{split}
		f'(\xi) &= -\mu \cdot\frac{\bigl(b^h+1\bigr)h\xi^{h-1}}{\left(b^h+{\xi}^h\right)^2}, \\
		f''(\xi) &= \mu\cdot\frac{h\xi^{h-2}\bigl(b^h+1\bigr)}{\bigl(b^h+\xi^h\bigr)^3}
					\Bigl((1+h)\xi^h-b^h(h-1)\Bigr).
	\end{split}
\]
Note, if $\displaystyle b=\sqrt[h]{\frac{1+h}{h-1}}$, then the point $\xi=1$ is an inflection point of $f$.
\begin{prop}\label{prop:hes1przegiecie}
	Let $h>1$ and
	\begin{equation}\label{war:hes1:przegiecie}
			\frac{k\,k_pk_r}{\alpha\beta} = \frac{h-1}{2h}\sqrt[h]{\frac{h-1}{h+1}}. 
	\end{equation}
	The positive steady state of~\eqref{hes1dimles}, with $f$ defined by~\eqref{fun_jen} and~\eqref{zamiana} 
	is globally asymptotically stable if and only if $h\le 3$. Moreover, the solution 
	$\bar p$ to~\eqref{nastac} is of the form
	\begin{equation}\label{ss:przegiecie}
		\bar p = k \sqrt[h]{\frac{h-1}{h+1}}
	\end{equation}
\end{prop}
\begin{proof}
	Due to~\eqref{war:hes1:przegiecie}, $\bar p$ given by~\eqref{ss:przegiecie} is the unique solution 
	to~\eqref{nastac}. Therefore, $\displaystyle b=\frac{k}{\bar p}=\sqrt[h]{\frac{1+h}{h-1}}$, and 
	the point $\xi=1$ is an inflection point of $f$. Hence, we may use Proposition~\eqref{prop:mal} getting 
	the condition
	\[
		|f'(1)| \le \mu  \; \Longleftrightarrow \; 
		\frac{h}{b^h+1} \le 1.
	\]
	Since $h>1$, an easy calculation shows inequality $h\le b^h+1$ is equivalent to $h\le 3$. 
\end{proof}

\begin{prop}\label{pro:stabb}
	Let $\bar b$ be a~unique positive solution to 
	\begin{equation}\label{war:stab:b}
		\frac{(b^h+1)h\big(\xi_0(b)\big)^{h-1}}{b^h+\big(\xi_0(b)\big)^h} = 1,
	\end{equation}
	where $\xi_0(b)$ is a~positive solution to
	\begin{equation}\label{na:xodb}
		\xi^{2h} +\xi^h\Bigl(b^h-1-h(b^h+1)\Bigr) + \xi^{h-1}h(b^h+1)-b^h = 0 ,
	\end{equation}
	such that $\xi( )=1$ only if~\eqref{na:xodb} has a~triple root at $\xi=1$. In other cases 
	as $\xi(b)$ we choose a~root of~\eqref{na:xodb} different from $1$. 
	Then, if  
	\begin{equation}\label{warnab}
		\frac{\bar b^{h+1}}{\bar b^h+1}< \frac{k\,k_pk_r}{\alpha\beta} ,
	\end{equation}
	then the positive steady state of Eq.~\eqref{hes1dimles} with $f$ given by~\eqref{fun_jen}
	and~\eqref{zamiana} is globally asymptotically stable in $\Cb\Bigl([0,+\infty)^2\bigr)$. 
\end{prop}
\begin{proof}
	The proof is an easy application of Proposition~\ref{prop:mal}.  
	For $h>1$ the function~$f$ has exactly one inflection point for $\xi>0$ and  
	fulfils assumptions of Proposition~\ref{prop:mal}. 
	Tedious calculations lead to the conclusion that equality $g(\xi)=0$ is equivalent 
	to~\eqref{na:xodb}.
	
	In order to show that Eq.~\eqref{war:stab:b} has exactly one positive 
	solution we make the following change of variables:
	\[
		z = \xi^h, \quad a~= b^h+1.
	\]
	Then Eq.~\eqref{war:stab:b} reads 
	\begin{equation}\label{war:glob:stab:1}
		\frac{ahz^{1-\frac{1}{h}}}{a-1+z}=1,
	\end{equation}
	where $z$ and $a$ are related by the equality 
	\begin{equation}\label{na:ztoa}
		z^2 + (a(1-h)-2)z +ha z^{1-\frac{1}{h}} +1-a = 0. 
	\end{equation}
	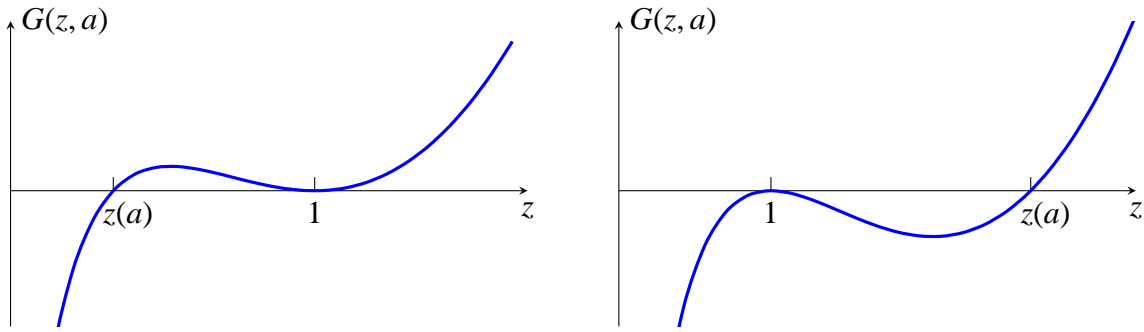
\begin{figure}[!htb]
	\centerline{
	\begin{tikzpicture}[>=stealth]
		\begin{scope}[xscale=4,yscale=9]
			\draw[->] (0,0) -- (1.7,0) node[below] {$z$};
			\draw[->] (0,-0.2) -- (0,0.25) node[right] {$G(z,a)$};
			\draw (1,0) node[below] {$1$} -- +(0,0.02);
			\draw (0.337722,0) node[below right] {\!\!\!\!$z(a)$} -- +(0,0.02);
				\clip (0,-0.2) rectangle (1.7,0.25);
				\draw[blue, very thick, domain=0:1.65,smooth] plot (\x, {\x^2-4.5*\x+5*sqrt(\x) - 1.5});
		\end{scope}
		\begin{scope}[xshift=8cm,xscale=2,yscale=3]
			\draw[->] (0,0) -- (3.4,0) node[below] {$z$};
			\draw[->] (0,-0.6) -- (0,0.75) node[right] {$G(z,a)$};
			\draw (1,0) node[below] {$1$} -- +(0,0.06);
			\draw (2.7085,0) node[below right] {\!\!\!\!$z(a)$} -- +(0,0.06);
				\clip (0,-0.6) rectangle (3.4,0.75);
				\draw[blue, very thick, domain=0:3.4,smooth] plot (\x, {\x^2-9*\x+14*sqrt(\x) - 6});
		\end{scope}
	\end{tikzpicture}
	}
	\caption{The possible graphs of $G(z,a)$ for fixes $a$. In the left-hand side
	panel the case when $z(a)<1$ while the case $z(a)>01$ is presented on the
	right-hand side panel. \label{fig.fun-g}}
	\end{figure}
	First we show the function $z(a)$, where $z$ is the solution to~\eqref{na:ztoa} is 
	a decreasing function of $a$. Let us denote the left-hand side of~\eqref{na:ztoa} as
	$G(z,a)$. Note, that substitutions made does not change number of roots. Thus,  
	due to Proposition~\ref{prop:mal} there exists a~multiple root of~\eqref{na:ztoa} at 
	$1$ and for $a>1$ there exits one more root in $(0,+\infty)$ as $G(0,a)<0$. Moreover, 
	as $G(z,a)\to+\infty$ as $z\to+\infty$, we can deduce that $\partial G(z(a),a)/\partial z>0$ (see Fig.~\ref{fig.fun-g}). 
	The Implicit Function Theorem implies that the sign of $z'(a)$ is reverse to the sign of 
	$\partial G(z,a)/\partial a$ at $z=z(a)$. An easy calculation leads to the following formula
	\[
		\frac{\partial G(z,a)}{\partial a} = z(1-h)+h z^{1-\frac{1}{h}}-1. 
	\]
	As $h>1$ we have $\frac{\partial G(0,a)}{\partial a}<0$, 
	$\frac{\partial G(z,a)}{\partial a}<0$  for $z$ sufficiently large and 
	$\frac{\partial G(1,a)}{\partial a}=0$. Note, 
	\[
		\frac{\partial^2 G(z,a)}{\partial a~\partial z} = (1-h)+(h-1) z^{-\frac{1}{h}}
			= z^{\frac{1}{h}}(h-1)\Bigl(1-z^{\frac{1}{h}}\Bigr).
	\]
	This implies that $\partial G(z,a)/\partial a$ is an increasing function of 
	$z$ for $0<z<1$ and is a~decreasing function of $z$ for $z>1$. This proves 
	$\partial G(z,a)/\partial a<0$ and therefore, we conclude $z'(a)>0$ for all 
	$a$ such that $z(a)\neq 1$. Thus, the function $z(a)$ is increasing and 
	there exists an inverse function $a(z)$. 

	Now, we calculate $a$ as a~function of $z$. From~\eqref{na:ztoa} we have 
	\[
		a = \frac{(z-1)^2}{1+(h-1)z-hz^{1-\frac{1}{h}}}.
	\]
	Plugging this $a$ into~\eqref{war:glob:stab:1} we obtain that the numerator 
	of~\eqref{war:glob:stab:1}	is
	\begin{equation}\label{licznik}
		\frac{h(z-1)^2z^{1-\frac{1}{h}}} {1+(h-1)z-hz^{1-\frac{1}{h}}}
	\end{equation}
	while the denominator reads 
	\begin{equation}\label{mianownik}
		\frac{z^2-2z+1-(1+(h-1)z-hz^{1-\frac{1}{h}})+z(1+(h-1)z-hz^{1-\frac{1}{h}})}{1+(h-1)z-hz^{1-\frac{1}{h}}}
		= \frac{hz(1-z^{-\frac{1}{h}})(z-1)}{1+(h-1)z-hz^{1-\frac{1}{h}}}
	\end{equation}
	Thus, the left-hand side of \eqref{war:glob:stab:1} reads 
	\begin{equation}\label{zixi}
		\frac{h(z-1)^2z^{1-\frac{1}{h}}}{hz(1-z^{-\frac{1}{h}})(z-1)}=
		\frac{z-1}{z^{\frac{1}{h}}-1} 
		= \frac{\xi^h-1}{\xi-1} .
	\end{equation}
	It is easy to see that this expression is an increasing function of $\xi$ and thus of $z$ 
	(the function can be in a~continuous way extended for $z=\xi=1$). Moreover, 
	\eqref{zixi} has limit $1$ at $z=0$ and $+\infty$ as $z\to+\infty$. Therefore, 
	for any $a>1$ there exists a~unique solution to~\eqref{war:glob:stab:1}. 
	This yields that there exists a~unique solution $\bar b$ to~\eqref{war:stab:b} and 
	that left-hand side of~\eqref{war:stab:b} is a~decreasing  function of $b$. 
	
	Due to Proposition~~\ref{prop:mal}, the steady state of Eq.~\eqref{hes1dimles}
	is globally stable for $b<\bar b$. 
	Because $b=k/\bar p$, we deduce if $b\le \bar b$, then
	$k/\bar b <\bar p $, and~\eqref{warnab} follows from the fact that the 
	left-hand side of~\eqref{nastac} is a~decreasing function of $\xi$. 	
\end{proof}

From the proof of Proposition~\ref{prop:mal} we deduce that the left-hand side of~\eqref{na:xodb}
has double root at $\xi=1$. Therefore, for $h=2$ we can easily calculate $\xi_0$ and give 
an explicit formula for~\eqref{warnab}.
\begin{prop}
	For $h=2$ if 
	\begin{equation}\label{warglobh2}
		\frac{k\,k_pk_r}{\alpha\beta} \ge \frac{5\sqrt{5}}{18} \approx 0.6211 ,
	\end{equation}
	then the positive steady state of Eq.~\eqref{hes1dimles} with $f$ given by~\eqref{fun_jen}
	and~\eqref{zamiana} is globally asymptotically stable in $\Cb\Bigl([0,+\infty)^2\bigr)$. 
\end{prop}
\begin{proof}
	For $h=2$, Eq.~\eqref{na:xodb} reduces to
	\[
		(\xi-1)^2\Bigl(\xi^2+2\xi-b^2\Bigr)=0 \; \Longrightarrow \; \xi_0 = -1+\sqrt{1+b^2},
	\]
	and global stability condition from Proposition~\ref{prop:mal} reads 
	\[
		|f'(\xi_0)|\le \mu \; \Longrightarrow \; 
		\frac{\bigl(b^2+1\bigr)2\xi_0 }{\left(b^2+{\xi_0}^2\right)^2} \le 1
	\]
	Plugging $\xi=  -1+\sqrt{1+b^2}$ into the above inequality, after some calculation 
	we arrive at
	\begin{equation}\label{h2warbpos1}
		2-b^2 \le 2(b^2-1)\sqrt{b^2+1}
	\end{equation}
	It is easy to see that the inequality~\eqref{h2warbpos1} is true for $b\ge \sqrt{2}$ 
	and it is false for $b\le 1$. Assume then $1<b<\sqrt{2}$. Squaring~\eqref{h2warbpos1},
	and after some easy calculation we deduce that for $1<b<\sqrt{2}$, the inequality
	\eqref{h2warbpos1} is equivalent to 
	\[
		b\ge \frac{\sqrt{5}}{2}.
	\]
	As $\sqrt{5}/2<\sqrt{2}$ we deduce that the inequality~\eqref{h2warbpos1}
	is equivalent to $b\ge \sqrt{5}/2$.  This implies $\bar p \le 2k/\sqrt{5}$.
	Since $\bar p$ is a~solution to~\eqref{nastac}, and left-hand side of~\eqref{nastac} is 
	a decreasing function of $\xi$, we have 
	$\bar p \le 2k/\sqrt{5}$ as long as 
	\[
		\frac{\alpha k^2}{k^2 + \Bigl(\frac{2k}{\sqrt{5}}\Bigr)^2} 
				\ge \frac{2k\,k_pk_r}{\beta\sqrt{5}},
	\]
	which is equivalent to~\eqref{warglobh2}.
\end{proof}

For $h\neq 2$ it is difficult do obtain similar result as for $h=2$. In Fig.~\ref{fig:deponh} 
we illustrated global stability region of the positive steady state of Eq.~\eqref{hes_sys_gen}
with $\tilde f$ given by~\eqref{fun_jen} in dependence on the Hill coefficient $h$. 

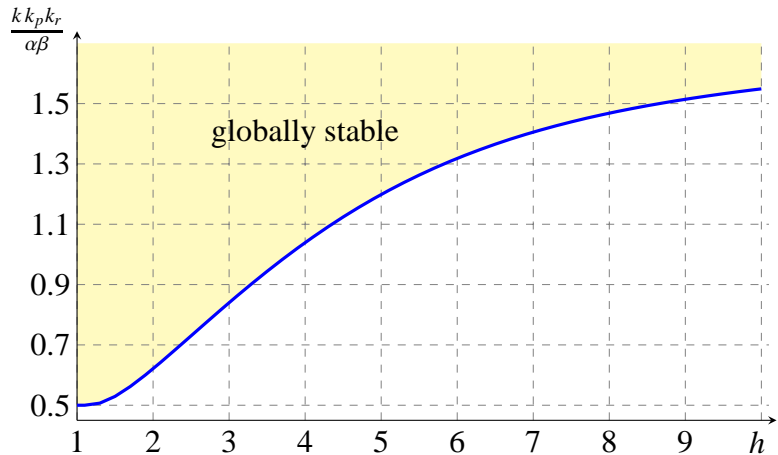
\begin{figure}[!hbt]
\centerline{
\begin{tikzpicture}[>=stealth,yscale=4]
	\def\ogr{(1,0.5) -- (1.0500,0.5000)	 -- (1.0900,0.5000)	 -- (1.1000,0.5000)
	 -- (1.3000,0.5069)	 -- (1.5000,0.5293)	 -- (1.7000,0.5620)	 -- (1.9000,0.6005)	 -- (2.1000,0.6423)
	 -- (2.3000,0.6859)	 -- (2.5000,0.7301)	 -- (2.7000,0.7744)	 -- (2.9000,0.8183)	 -- (3.1000,0.8614)
	 -- (3.3000,0.9033)	 -- (3.5000,0.9440)	 -- (3.7000,0.9832)	 -- (3.9000,1.0209)	 -- (4.1000,1.0569)
	 -- (4.3000,1.0913)	 -- (4.5000,1.1241)	 -- (4.7000,1.1551)	 -- (4.9000,1.1845)	 -- (5.1000,1.2122)
	 -- (5.3000,1.2384)	 -- (5.5000,1.2630)	 -- (5.7000,1.2862)	 -- (5.9000,1.3079)	 -- (6.1000,1.3283)
	 -- (6.3000,1.3474)	 -- (6.5000,1.3654)	 -- (6.7000,1.3822)	 -- (6.9000,1.3979)	 -- (7.1000,1.4126)
	 -- (7.3000,1.4264)	 -- (7.5000,1.4393)	 -- (7.7000,1.4515)	 -- (7.9000,1.4628)	 -- (8.1000,1.4735)
	 -- (8.3000,1.4835)	 -- (8.5000,1.4930)	 -- (8.7000,1.5018)	 -- (8.9000,1.5102)	 -- (9.1000,1.5181)
	 -- (9.3000,1.5256)	 -- (9.5000,1.5327)	 -- (9.7000,1.5394)	 -- (9.9000,1.5458)	 -- (9.9000,1.5458)
	 -- (10,1.5488)}
	\fill[yellow!30!white] \ogr -- (10,1.7) -- (1,1.7) -- cycle;
	\draw[help lines,dashed, ystep=0.2, yshift=0.1cm] (1,0.35) grid (10.1,1.59);
	\draw[->] (1,0.45) -- (10.2,0.45) node[below left] {$h$};
	\draw[->] (1,0.45) -- (1,1.74) node[left] {$\frac{k\,k_pk_r}{\alpha\beta}$};
	\foreach \i in {1,...,9}
		\node at (\i,0.38) {$\i$};
	\foreach \i in {0.5,0.7,0.9,1.1,1.3,1.5}
		\node at (0.65,\i) {$\i$};
	\draw[very thick,blue] \ogr ;
	\node at (4,1.4) {globally stable};
\end{tikzpicture}
}
\caption{The dependance of the critical value of $k\,k_pk_r/(\alpha\beta)$ on  
	Hill coefficient $h$. The region above the curve denotes the global stability 
	region calculated on the basis of Proposition~\ref{pro:stabb}.\label{fig:deponh}}
\end{figure}

\begin{rem}
	For $h=1$ the functions $\tilde f$ given by~\eqref{fun_jen} as well as 
	$f$ given by~\eqref{zamiana}, are convex for all 
	$\xi\ge 0$. From Proposition~\ref{prop:mal} we deduce the slope $\alpha_1$
	of the line that bounds the graph of $f$ from above on $(0,1)$ is equal to
	$f(0)-f(1)=\mu/b$. This yields, that if $b\ge 1$ which is equivalent to  
	$k\,k_pk_r/(\alpha\beta)\ge 1/2$, the positive steady state of Eq.~\eqref{hes1dimles}
	is globally asymptotically stable in $\Cb\Bigl([0,+\infty)^2\bigr)$.
\end{rem}

\section{Discussion} \label{sec:disc}

In this paper, we gave the condition for global stability of zero steady state 
for a~class of delay differential equations. The method used in the paper is based 
on the theory developed by Liz and Ruiz-Herrera~\cite{Liz2013jde}. In this paper we applied 
this method to the 
system of equations of the particular form, that can describe a~cascade of chemical reactions, 
such that $j$-th chemical is produced by $j-1$-th and the last one affects production 
of the first one. This is a~generalisation of the Hes1 gene expression model proposed 
by Monk~\cite{monk03currbiol}. 

We assumed the function $f_j$ can depend on time. In fact, here we can also include the 
dependence of $f_j$ on other variables of the system. If we denote $x=(x_1,\dotsc,x_k)$ 
and $x_t:[-\tau,0]\to\R^k$ be a~function defined by $x_t(s)=x(t+s)$ for $s\in[-\tau,0]$,
we can consider $f_j(t,x_t(s),x_{j-1}(t+s))$ instead of $f_j(t,x_{j-1}(t+s))$. However,
we need then a~very strong assumption of $f_j$, that is~\eqref{warunkig} need to be 
replaced by 
\begin{equation}\label{warunkig1}
	|f_j(t,\varphi,x)| \le \frac{\alpha_j}{\mu_j}|x|, \quad x\in\R, \;\; j=2,3,\dotsc,k, 
			\qquad
	|f_1(t,\varphi,x)| < \frac{\alpha_1}{\mu_1}|x|, \quad x\neq 0, 
\end{equation}
for all $t\ge t_0$, and all $\varphi \in \Cb(\R^k)$. In particular this assumption yields 
$f_j(t,\varphi,0)=0$ for any $\varphi\in\Cb(\R^k)$, which would not be generally true. However, 
if $f_j(t,x_t,x_{j-1}(t)) = \tilde f_j(t,x_{j-1}(t)) \tilde g_j(x_j(t))$, for some 
globally bounded function $\tilde g_j$ and $\tilde f_j$ fulfilling~\eqref{warunkig}, with
$\alpha_j$ replaced by $\tilde\alpha_j$, then~\eqref{warunkig1} is fulfilled with 
$\alpha_j = \tilde\alpha_j \sup |\tilde g(x)|$. 

Here, we give two simple examples. First, we recall global stability result for 
a vector disease model derived by Cooke~\cite{Cooke1979}. Later, we justify, 
that the method derived in this paper cannot be applied to the p53-Mdm2 model proposed 
by Monk~\cite{monk03currbiol}. The model considered in~\cite{Cooke1979} is the following 
\begin{equation}\label{vectord}
	\dot x(t) = bx(t-\tau)\bigl(1-x(t)\bigr)-cx(t).
\end{equation}
In~\cite{Cooke1979}, by constructing Liapunov functionals, it was proved that if $b\le c$ then 
$0$ is globally asymptotically stable in $\Cb([0,+\infty))$ while for $b>c$ the steady state 
$\bar x = 1-c/b$ is globally asymptotically stable in $\Cb((0,+\infty))$. We show, that 
the technique developed here can be applied to~\eqref{vectord}. In fact, in~\cite{Cooke1979}  
it was shown, that the set $\Cb([0,1])$ is absorbing. Therefore, we can estimate
\[
|bx(t-\tau)\bigl(1-x(t)\bigr)| \le b x(t-\tau) \text{CZY TU~NIE POWINNO~BYÆ} b<c???
\]
for any $x(t)\in [0,1]$. Thus, using Theorem~\ref{thm:globalstability} we deduce 
global stability of $0$ in $\Cb([0,1])$ and thus in $\Cb([0,\infty))$. A~similar arguments yield 
global stability of $1-c/b$ for $c<b$. 

On the other hand, the first equation of the p53-Mdm2 model proposed by 
Monk~\cite{monk03currbiol}, after moving the positive steady state $(\bar x, \bar y, \bar z)$ 
to $(0,0,0)$ reads
\[
	\dot x = 1-\mu_2\frac{(z+\bar z)^2}{z_0+(z+\bar z)^2}(x+\bar x)
		-\mu_1\bar x -\mu_1 x,
\]
where $z$ is the third variable of the model. It is easy to see, that for $z\neq 0$ and 
$x=0$, we have 
\[
	1-\mu_2\frac{(z+\bar z)^2}{z_0+(z+\bar z)^2}(x+\bar x)-\mu_1\bar x\neq 0
\]
and therefore~\eqref{warunkig1} cannot be fulfilled.

\section*{Acknowledgements}
The paper was supported by the Polish Ministry of Science and Higher
Education, within the Iuventus Plus Grant: "Mathematical modelling of neoplastic processes"
grant No. IP2011 041971.

\smallskip{}
\noindent
The author whish to thank Prof. Urszula Fory\'s for reading a~draft version and many valuable comments.


\end{document}